\documentclass[12pt]{amsart}
\usepackage{hyperref}
\date{}

\newtheorem{theorem}{Theorem}
\newtheorem{lemma}[theorem]{Lemma}
\newtheorem{proposition}[theorem]{Proposition}
\newtheorem{question}[theorem]{Question}

\newtheorem{remark}[theorem]{Remark}

\def\SS{\pm^{\mathbb N}}

\begin{document}
\title[The fair soup division and approximating numbers]
{The fair soup division and approximating numbers}

\author[Alex Ravsky]{Alex Ravsky}
\address{Pidstryhach Institute for Applied Problems of Mechanics and Mathematics
National Academy of Sciences of Ukraine}
\email{alexander.ravsky@uni-wuerzburg.de}

\begin{abstract}
We consider a recent The Vee's fair soup division problem,
provide its partial solution, and pose a related open problem.
\end{abstract}

\dedicatory{Dedicated to the personnel of canteen at Karpenko Physico-Mechanical Institute,
Lviv, Ukraine}

\maketitle \baselineskip15pt

Food division problems were always interesting for mathematicians, especially before a dinner.
For instance, we can divide cakes~\cite{KK}, \cite{Rav}, \cite{S-HS}, \cite[Prob. 51]{Sht}, \cite{use},
pancakes~\cite[$\S$ 11]{CS}, \cite[$\S$ 7]{Sha}, pizzas~\cite{Wik}, ham sandwiches~\cite{Wik2},
and fried eggs~\cite{Tyr}.

This paper is devoted to the following Soup Division Problem $SDP(q)$ (where $0<q<1$), posed recently by
The Vee on Mathematics StackExchange~\cite{Vee}. We are given an unbounded bowl of soup,
a bounded ladle, and two plates, marked "$+$" and "$-$". We consecutively scoop a full ladle of
the soup from the bowl and pour it to one of plates.
The soup contains two nutritious stuffs. The first of them
is evenly dissolved in the volume; the second is floated to the surface and
every scoop takes a part $1-q$ of it, so its amount goes down geometrically with a quotient of $q$.

Let $\pm^{\mathbb N}$ be a set of sequences of signs "$+$" and "$-$".
We say that $(*)_i\in\SS$ is an
\emph{asymptotically fair division} for $SDP(q)$ if an order $(*)_i$ of plates to pour
provides asymptotically equal shares of both stuffs to them. That is,
$\sum_{i=1}^n *_i1=o(n)$ and $\lim_{n\to\infty} \sum_{i=1}^n *_iq^i =0.$
Moreover, $(*)_i$ is a \emph{boundedly fair division}, if $\sum_{i=1}^n *_i1=O(1)$.

There is the following natural

\begin{question}\label{q1} For which numbers $q\in (0,1)$, $SDP(q)$ has an asymptotically (boundedly)
fair division?
\end{question}

Answering it, we consider periodic divisions first. It is easy to show the following

\begin{proposition}\label{periodic} Let $0<q<1$ and $(*)_i\in\SS$ be a sequence with a period $N$.
The following conditions are equivalent:
\begin{itemize}
\item $(*)_i$ is an asymptotically fair division for $SDP(q)$;
\item $(*)_i$ is a boundedly fair division for $SDP(q)$;
\item $\sum_{i=1}^N *_i1= \sum_{i=1}^N *_iq^i =0$.
\end{itemize}
\end{proposition}

It is easy to check that the shortest period of a
periodical asymptotically fair division for $SDP(q)$ is $6$ (for instance,
for a sequence $+---++\dots$ and $q=\varphi^{-1}$, where $\varphi$ is the golden ratio).

The last condition of Proposition~\ref{periodic} suggests the following notion.
A \emph{balanced $\pm$-polynomial} is a polynomial $P(x)$ of the form
$\sum_{i=1}^n \pm x^i$ such that $P(1)=0$. Then for each $q\in (0,1)$, by Proposition~\ref{periodic},
$SDP(q)$ has a periodic asymptotically fair division iff $q$ is a root of a balanced $\pm$-polynomial.



Now we consider general fair divisions.
If $0<q\le 1/2$ then $q\ge\sum_{i=2}^\infty q^i$, so $SDP(q)$ has no asymptotically fair division.

To obtain a partial answer for Question~\ref{q1}, we need the following

\begin{lemma}\label{lem1}
Let $(a_n)$ be an absolutely convergent series of real numbers such that for any natural $k$ we have
$$|a_{2k-1}-a_{2k}|\le \sum_{n=k+1}^\infty |a_{2n-1}-a_{2n}|.\label{1}\eqno{(1)}$$
Then there exists a sequence $(*)_i\in \pm^{\mathbb N}$ of signs
such that  $*_{2k-1}1+*_{2k}1=0$ for each natural $k$
and $\sum_{i=1}^\infty *_ia_i=0$.
\end{lemma}
\begin{proof} For each natural $k$ put $b_k=|a_{2k-1}-a_{2k}|$. It suffices to show that we can
consecutively choose signs for $b_k$, providing $\sum_{k=1}^\infty \pm b_k=0$. At the beginning choose
a sign "$+$" for $b_1$ and for each $k>1$ choose for $b_k$ a sign "$+$",
 if $\sum_{i=1}^{k-1} \pm b_i\ge 0$ and "$-$", otherwise.
It is easy to check that (1) provides $\sum_{k=1}^\infty \pm b_k=0$.
\end{proof}

\begin{proposition}\label{prop1} For each $q\in [1/\sqrt{2},1)$, $SDP(q)$ has a boundedly fair division.
\end{proposition}
\begin{proof} By Lemma~\ref{lem1}, it suffices to check that for any any natural $k$ we have
$$q^{2k-1}-q^{2k}\le \sum_{i=k+1}^\infty q^{2i-1}-q^{2i}=\frac{q^{2k+1}}{1+q}.$$
\end{proof}

To improve Proposition~\ref{prop1}, we introduce the following notion.
A number $q\in (1/2, 1)$ is \emph{approximating}, if there exist non-negative numbers $A$ and $N$
such that for each $x_0\in [0,A]$ there exist a balanced $\pm$-polynomial $P$ of degree $n\le N$
such that $|x_0-P(q)|\le Aq^n$.

\begin{proposition} For each approximating number $q$, $SDP(q)$ has a boundedly fair division.
\end{proposition}
\begin{proof} Suppose that the number $q$ is approximating with the constants $A$ and $N$. Put $k_0=0$.
Then we can inductively build an increasing sequence $(k_n)$ of natural numbers such that
$k_{n+1}-k_n\le N$ for each $n$, assigning signs $*_i$ to numbers $q^i$ for $k_{n-1}<i\le k_n$
(a half of the assigned signs are "$+$" and the other half are "$-$") assuring
$\sum_{i=1}^{k_n} \pm q^i\le Aq^{k_n}$.  The construction assures that $(*)_i$ is
a boundedly fair division for $SDP(q)$.
\end{proof}

This suggests the following

\begin{question} Which numbers $q\in (1/2,1)$ are approximating?
\end{question}

The following proposition provides a partial answer to it.
Let $q_\infty=0.5845751\dots$ be a unique positive root of a polynomial $Q_\infty(x)=x^4+x^3+2x^2-1$.

\begin{proposition}\label{prop:qinfty} Each number $q\in (q_\infty,1)$ is approximating.
\end{proposition}
\begin{proof} For each natural $n$ put
$$P_n(x)=x-x^{2n}+\sum_{i=2}^{2n-1} (-x)^i=x-x^{2n}+\frac{x^2-x^{2n}}{1+x}.$$
Clearly, $P_n(x)$ is a balanced $\pm$-polynomial. Given a natural number $N$, put
$A=\tfrac {P_N(q)}{1-q^{2N}}$. A segment $[0,A]$ is covered by a family $$[0, P_1(q)],
[P_1(q),P_2(q)],\dots, [P_{N-1}(q),P_N(q)], [P_N(q), A]$$ of segments. So
in order to show that for each $x_0\in [0,A]$ there exist $1\le n\le N$ such that
$|x_0-P_n(q)|\le Aq^{2n}$ it suffices to show that that each point of a segment
of the family lies sufficiently close to a suitable endpoint of the segment.
Namely, that $Aq^2\ge P_1(q)$,
$|P_{n+1}(q)-P_n(q)|\le A(q^{2n}+q^{2n+2})$ for each $1\le n\le N-1$, and $P_N(q)+Aq^{2N}\ge A$.

The last inequality follows from the definition of $A$. It is easy to check that for
each $1\le n\le N-1$ we have
$$\frac{|P_{n+1}(q)-P_n(q)|}{q^{2n}+q^{2n+2}}=\frac{2-q-q^2}{1+q^2}.$$
Put $P_\infty(x)=x+\frac{x^2}{1+x}$. When $N$ tends to the infinity,
$P_\infty(q)-P_N(q)/(1-q^{2N})$ tends to zero. Since we can choose $N$ arbitrarily big, it suffices to
show that
$$\frac{2-q-q^2}{1+q^2}<P_\infty(q),$$
that is $Q_\infty(q)>0$, which holds because $q>q_\infty$.
Finally, the inequality $Aq^2\ge P_1(q)$ for all sufficiently big $N$ holds because
$$(q+q^2)\left(P_\infty(q)-\frac{P_1(q)}{q^2}\right)=-1+2q^2+2q^3>-1+2q^2_\infty+2q^3_\infty>0.$$
\end{proof}

\begin{remark} The approach from the proof of Proposition~\ref{prop:qinfty} cannot
provide that numbers not bigger than $1/\sqrt{3}=0.5773\dots$ are approximating. Indeed,
suppose that for a given $q\in (1/2,1)$ there exist a positive number $A$ and
balanced $\pm$-polynomials $P_1,\dots, P_n$ of degrees $2,\dots,2n$ respectively
such that for each $x_0\in [0,A]$ there exist $1\le i\le n$ such that $|x_0-P_i(q)|\le Aq^{2i}$.
Then
$$A\le 2A(q^2+q^4+\dots+q^{2n})<2A(q^2+q^4+\dots)=A\frac {2q^2}{1-q^2},$$
which follows $q>1/\sqrt{3}$.

\end{remark}

\end{document}